\documentclass{amsart}
\usepackage{amssymb,enumerate}
           
\usepackage{tikz-cd}
\usepackage{hyperref}
\hypersetup{%
  bookmarksnumbered=true,%
  colorlinks=true,%
  linkcolor=blue,%
  citecolor=blue,%
  filecolor=blue,%
  menucolor=blue,%
  urlcolor=blue,%
  bookmarksopen=true,%
  bookmarksdepth=2,%
  pageanchor=true}

\usepackage{mathtools}
\usepackage{todonotes}

\numberwithin{equation}{section}

\theoremstyle{plain}

\newtheorem{theorem}{Theorem}[section]
\newtheorem{proposition}[theorem]{Proposition}

\theoremstyle{definition}
\newtheorem{definition}[theorem]{Definition}
\newtheorem{example}[theorem]{Example}

\newtheorem*{ack}{Acknowledgements}

\theoremstyle{remark}
\newtheorem{remark}[theorem]{Remark}

\newtheorem{chunk}[theorem]{}

\newcommand{\edim}{\operatorname{emb\,dim}}
\newcommand{\End}{\operatorname{End}}
\newcommand{\Ext}{\operatorname{Ext}}
 \newcommand{\fm}{\mathfrak{m}}
\newcommand{\Hom}{\operatorname{Hom}}
\newcommand{\Ker}{\operatorname{Ker}}
\newcommand{\rank}{\operatorname{rank}}

\begin{document}

\title[Rigid ideals]{Some non-principal rigid ideals in \\ Gorenstein
  domains of dimension one}

\author[Christensen]{Lars Winther Christensen} %
\address{Texas Tech University, Lubbock, TX 79409, U.S.A.}
\email{lars.w.christensen@ttu.edu}
\urladdr{https://larswinther.github.io/homepage/}

\author[Gerko]{Alex Gerko}
\address{XTX Markets, R7 14-18 Handyside Street London,
  N1C 4DN, U.K.}  \email{alexander.gerko@xtxmarkets.com}

\author[Iyengar]{Srikanth B. Iyengar} \address{
  University of Utah\\
  Salt Lake City, UT 84112\\
  U.S.A.}  \email{srikanth.b.iyengar@utah.edu}

\begin{abstract}
  We discuss an example of a rigid non-principal ideal in a one
  dimensional (commutative) Gorenstein domain, which contradicts a
  conjecture of C.~Huneke and R. Wiegand. The construction and its
  analysis were discovered by Codex, when prompted by one of the
  authors to verify the conjecture or find a counterexample.  A proof
  of the conjecture, also discovered by Codex, is presented when the
  ring is equicharacteristic and its embedding dimension is at most
  three.
\end{abstract}

\date{21 August 2026}

\keywords{Torsion in tensor products, Gorenstein ring, rigid ideal}

\subjclass[2020]{13C13 (primary); 13D07 (secondary)}

\maketitle

\section{Introduction}

\noindent A conjecture of Huneke and Wiegand~
\cite[pp. 473]{Huneke/Wiegand:1994} postulates that any ideal $I$ in a
Gorenstein (commutative, noetherian) local domain $R$ with the
property that the $R$-module $I\otimes_R\Hom_R(I,R)$ is torsion-free,
must be free. The condition the tensor product is torsion-free is
equivalent to the condition that
\[
  \Ext^1_R(I,I)=0\,,
\]
that is to say, that the ideal $I$ is \emph{rigid}; see, for instance,
\cite[Proposition~4.3]{Huneke/Iyengar/Wiegand:2019}. This conjecture
has been verified in various special cases; see, for instance,
\cite{Huneke/Iyengar/Wiegand:2019}.

In this note we describe a general method, discovered by ChatGPT, for
constructing counter-examples to this conjecture; see
Propositions~\ref{prop:gorenstein} and \ref{prop:rigid}. A concrete
example realizing this method is given in
Example~\ref{ex:the-main-example}. In fact, in this example the ring
$R$ is standard-graded, and the ideal $I$ is two-generated, by linear
forms; we refer to Example~\ref{ex:the-main-example} for other
noteworthy features.

At about the same time as these examples were discovered, Son
Pham~\cite{Pham:2026} found another example, using GPT-5.6 Pro; this
has been verified independently by Huneke~\cite{Pham:2026}. This
example is a semi-group ring, and hence also a graded ring, but it is
not standard-graded.

\subsection*{How this was found}
Sol Ultra on Codex, prompted to prove the Huneke--Wiegand conjecture
for Gorenstein rings or find counterexample, and separately prompted
to either prove it for complete intersections or find counterexample.
The manuscript is a streamlined version of text generated by Codex.

\begin{ack}
  The authors express their thanks to Craig Huneke for sharing his
  verification of Song Pham's example.
\end{ack}

\section{Gorenstein domains of dimension one}

\noindent Given a field $K$ and a symmetric bilinear form on a finite
dimensional $K$-vector space $V$, one has a finite dimensional
commutative graded $K$-algebra $K\oplus V\oplus K$, where the only
non-trivial product $V\times V\to K$ is given by the form.  Moreover,
this algebra is Gorenstein if and only if the form is not degenerate;
see, for instance, \cite[(3.15)]{Lam:1999a}.

This section gives a method for constructing Gorenstein domains of
Krull dimension one, allegedly inspired by the work of Eisenbud and
Popescu~\cite[Theorem~7.1]{Eisenbud/Popescu:2000}, and also
Stevens~\cite{Stevens:2026}.

\begin{definition}
  \label{defn:the-ring-R}
  Let $K$ be a field and $K\subset L$ a finite extension, such that
  its degree $n\coloneqq [L:K]$ is even.  Fix a $K$-subspace
  $H\subset L$ containing $K$ and of index $1$, that is to say,
  $\rank_KH=n-1$.

  Let $W$ be a $K$-linear subspace of $H$ satisfying the following
  conditions:
  \[
    K\subsetneq W\,, \qquad \rank_KW = \frac n2\,, \qquad
    \text{and}\qquad W^2 \subseteq H\,.
  \]
  Let $t$ be an indeterminate and $L[t]$ the standard graded
  $L$-algebra, so $|t|=1$.  Since $W^2\subseteq H$ holds, the graded
  $K$-vector subspace of $L[t]$ given by
  \[
    R\coloneqq K\oplus Wt \oplus Ht^2 \bigoplus_{i\ge 3} Lt^i
  \]
  is a $K$-subalgebra. Here are the principal features of this ring.
\end{definition}

\begin{proposition}
  \label{prop:gorenstein}
  The ring $R$ is graded Gorenstein (in particular, noetherian) domain
  of Krull dimension one. As a $K$-algebra it is generated in degrees
  one and two; if $W^2=H$, then it is standard graded.
\end{proposition}

\begin{proof}
  Clearly, one has inclusion of subrings $K[t]\subset R\subset
  L[t]$. Since $L/K$ is finite dimensional, $L[t]$ is finitely
  generated, and even free, as an $K[t]$-module. Thus, its
  $K[t]$-submodule $R$ is also finitely generated, and free. This
  justifies the claim that the ring $R$ is a noetherian domain and of
  Krull dimension one.

  In the remainder of the proof, it is useful to consider the
  $K$-linear functional $f\colon L\to L/H\cong K$, and the symmetric
  $K$-bilinear form
  \[
    \langle -,-\rangle \colon L \times L\longrightarrow K \quad
    \text{where $\langle x,y\rangle = f(xy)$}\,.
  \]
  This form is nondegenerate: Given $0\ne x\in L$ one has
  $\langle x,x^{-1}z\rangle \ne 0$ for $z\in L\setminus H$.

  For any $K$-subspace $U\subseteq L$ one has the orthogonal subspace
  \begin{equation}
    \label{eq:orth}
    U^{\perp}\coloneqq\{x\in L\mid \langle x,-\rangle=0 \text{ on } U\}\,.    
  \end{equation}
  This subspace satisfies $\rank_K U + \rank_K U^{\perp}
  =\rank_KL$. We claim
  \[
    W^{\perp} = W \qquad\text{and}\qquad H^{\perp} = K\,.
  \]
  Indeed, the hypothesis that $W^2\subseteq H$ is equivalent to
  $W\subseteq W^{\perp}$; equality holds because of the hypothesis
  that $\rank_KW=n/2$ and \eqref{eq:orth}, which justifies the
  equality on the left. As to the one on the right, again from
  \eqref{eq:orth} we deduce that $\rank_KH^{\perp}=1$. Since
  $K\subset H=\Ker f$, we conclude that $H^{\perp}=K$.

  Since $H^{\perp}=K$, for any element $w$ in $W\setminus K$ one has
  that $wH\not\subseteq H$. Thus as a $K$-algebra $R$ is generated by
  its elements of degree one and two; evidently, it is standard graded
  if also $W^2=H$.

  It remains to verify that $R$ is Gorenstein. Since $R$ is a domain
  any nonzero element, in particular $t$, is a not a zero divisor on
  it. A direct calculation yields an isomorphism of $K$-algebras
  \[
    R/Rt \cong K \oplus (W/K)t \oplus (H/W)t^2 \oplus (L/H)t^3\,.
  \]
  The elements of degree three are in the socle. We verify that these
  are the only ones; since $\rank_K(L/H)=1$ this is equivalent to
  $R/Rt$ Gorenstein; see \cite[3.2.10]{Bruns/Herzog:1998a}

  Suppose the residue class of $wt$ is in the socle of $R/Rt$. Then,
  from the structure of $R$, we get $wH\subseteq H$, that is to say
  $w$ in $H^{\perp}=K$. Then $wt=0$ in $R/Rt$. Similarly if the
  residue class of $ht^2$ is in the socle, then $hW\subseteq H$, so
  $h$ is in $W^{\perp}=W$, and hence $ht^2=0$ in $R/Rt$.

  This completes the verification that $R/Rt$ is Gorenstein. Since $t$
  is a homogeneous element and not a zero-divisor, we conclude that
  $R$ is Gorenstein; see \cite[3.1.19]{Bruns/Herzog:1998a}.
\end{proof}

\section{Rigid ideals}
In this section, we identify some conditions on the rings described in
the previous section that ensure that they contain rigid ideals that
are not principal.

In the setup of Definition~\ref{defn:the-ring-R}, fix an element
$\alpha$ in $W\setminus K$ and consider the ideal
\[
  I \coloneqq (t,\alpha t)R
\]
of $R$. This ideal is clearly not principal, and hence not free. The
claim is that, for a suitable choice of subspace $W$, the ideal $I$ is
rigid. Not every choice of $W$ works, as can be verified by a direct
check.

We focus now on a special situation when it does, leaving the general
analysis that leads to this choice to the end of this section.

\begin{definition}
  \label{defn:W}
  Let $L=K[\alpha]$, where $\alpha$ has degree $n\equiv 0\mod 4$, with
  $\alpha\notin K$, so $n\ge 4$, and $\alpha^n\in K^\times$. Thus
  $1,\alpha,\alpha^2,\dots,\alpha^{n-1}$ is a basis for $L$ over
  $K$. For $i\ge 0$, let $L_i$ denote the $K$-subspace of $L$ spanned
  by $1,\dots,\alpha^i$. Thus one has a filtration of $L$ by
  $K$-subspaces
  \[
    K=L_0\subset L_1\subset \cdots \subset L_{n-1} = L\,.
  \]
  Choose a $K$-subspace $V\subset L$ containing $L_0=K$ and such that
  the following conditions hold:
  \begin{enumerate}[\quad\rm(1)]
  \item $\rank_KV = n/4$;
  \item the sum $V + \alpha V + \alpha^2 V$ is direct;
  \item $V^2 = L_{n-4}$;
  \item $L= VL_{n-3}$.
  \end{enumerate}
  With $V$ as above, set $W\coloneqq V\oplus \alpha V$ and
  $H\coloneqq L_{n-2}$. Recall that
  \[
    R\coloneqq K\oplus Wt \oplus Ht^2 \bigoplus_{i\ge 3} Lt^i
  \]
  viewed as a graded $K$-subalgebra of $L[t]$.
\end{definition}

\begin{proposition}
  \label{prop:rigid}
  In the setup of Definition \ref{defn:W}, the ring $R$ is a standard
  graded Gorenstein domain of Krull dimension one. The ideal
  $(t,\alpha\,t)$ is rigid, but not free.
\end{proposition}

\begin{proof}
  The claims about the ring $R$ follows from
  Proposition~\ref{prop:gorenstein}; observe that $W^2=H$, by
  condition (3) in Definition~\ref{defn:W}.

  To verify that $(t,\alpha\, t)$ is rigid, we use the criterion
  \cite[Corollary~4.9]{Huneke/Iyengar/Wiegand:2019}; that statement is
  for local rings, but the argument applies equally well for
  homogeneous ideals in graded rings. Thus, we have to verify the
  equality:
  \[
    ((t):_R \alpha\, t)\cap ((\alpha\,t):_R t) = ((t):_R \alpha\, t)
    ((\alpha\,t):_R t)\,.
  \]
  To begin with, observe that
  \[
    ((\alpha\,t):_R t) = \alpha\cdot ((t):_R \alpha\, t)\,;
  \]
  keep in mind that $\alpha$ itself is not an element of $R$. Thus the
  desired equality is that
  \[
    J\cap \alpha J = \alpha J^2 \quad \text{where
      $J=((t):_R \alpha\, t)$.}
  \]
  Since $(t)= Kt \oplus Wt^2 \oplus Ht^3 \bigoplus_{i\ge 4} Lt^{i}$,
  using condition (2) in Definition~\ref{defn:W} a direct verification
  yields
  \[
    J = Vt \oplus L_{n-3}t^2 \bigoplus_{i\ge 3} Lt^i\qquad\text{and
      hence}\qquad \alpha J = \alpha V t \oplus \alpha L_{n-3}t^2
    \bigoplus_{i\ge 3} Lt^i\,.
  \]
  Observing that $V\cap \alpha V=0$, by condition (2) in
  Definition~\ref{defn:W}, one thus gets that
  \[
    J\cap \alpha J = (L_{n-3}\cap \alpha L_{n-3})t^2 \bigoplus_{i\ge
      3} Lt^i\,.
  \]
  On the other hand, a direct computation yields the first equality
  below
  \begin{align*}
    \alpha J^2 & = \alpha (V^2 t^2 \oplus V L_{n-3}t^3 \bigoplus_{i\ge 4} Lt^i) \\
               &=\alpha (L_{n-4}t^2 \oplus Lt^3 \bigoplus_{i\ge 4}Lt^i) \\
               &  = \alpha L_{n-4} t^2 \oplus \bigoplus_{i\ge 3} Lt^i\,.
  \end{align*}
  The second equality holds by conditions (3) and (4) in
  Definition~\ref{defn:W}, and the third follows from $\alpha L=L$.
  Thus, to verify that $J\cap\alpha\, J=\alpha J^2$, it remains to
  verify that
  \[
    L_{n-3}\cap \alpha L_{n-3} = \alpha L_{n-4}\,.
  \]
  This is clear, given the description of the $L_i$.
\end{proof}

\subsection*{A numerical constraint}
In the notation above, suppose $d\coloneqq \rank_KV$, so that
$\rank_KW = 2d$ and $\rank_KL =4d$. Thus the condition that
$V^2=L_{n-4}$, from Definition~\ref{defn:W}(2), implies that
\[
  \binom{d+1}{2} = \rank_K\mathrm{Sym}^2_K(V)\ge \rank_K L_{n-4} = 4d
  - 3\,.
\]
This leads to the condition that
\[
  (d-1)(d-6)\ge 0\,.
\]
One cannot have $d=1$, for then $V=K$ so $VL_{n-3} = L_{n-3} \ne
L$. Thus the smallest possible value for $d$ that would permit $V$ to
satisfy the constraints in Definition~\ref{defn:W} is $6$. Then
$\rank_KL=24$.

When $d=6$, one has
\[
  \rank_K \mathrm{Sym}^2_K(V) = 21 = \rank_K L_{20}
\]
so condition (2) in Definition~\ref{defn:W} is equivalent to the
non-vanishing of a determinant of size $21\times 21$.

These considerations motivate the following example.

\begin{example}
  \label{ex:the-main-example}
  Let $K=\mathbb{Q}$ and
  $L\coloneqq \mathbb{Q}[\alpha]/(\alpha^{24}-2)$. The polynomial
  $x^{24}-2$ is irreducible, by, say the Eisenstein criterion, so
  $\rank_KL=24$. Let $V$ be the $\mathbb{Q}$-subspace of $L$ spanned
  by the six elements
  \begin{gather*}
    1, \quad \alpha - \alpha^{15}, \quad \alpha^2 - \alpha^{12} - \alpha^{16}\\\
    \alpha^4 - \alpha^{14}, \quad \alpha^6 - \alpha^{13}+\alpha^{20},
    \quad \alpha^9+\alpha^{16}+\alpha^{19}\,.
  \end{gather*}
  One can verify directly by hand, or using a computer algebra system,
  that $V$ satisfies the conditions listed in Definition~\ref{defn:W};
  see Section~\ref{sec:M2} for a Macaulay 2 verification. Thus
  Propositions~\ref{prop:gorenstein} and \ref{prop:rigid} justify the
  following claims.

  With $V$ and $L$ as above, and $W\coloneqq V\oplus \alpha V$, the
  subring
  \[
    R\coloneqq \mathbb{Q} \oplus Wt \oplus L_{22}t^2
    \oplus_{i\geqslant 3} Lt^i
  \]
  of $L[t]$ is a standard graded Gorenstein $\mathbb{Q}$-algebra and a
  domain. Moreover, the ideal $I\coloneqq (t,\alpha t)$ is rigid, but
  not principal.

  The ideal $I$ has the added property that the evaluation map
  \[
    I\otimes_R \operatorname{Hom}_R(I,R)\longrightarrow R
  \]
  is an isomorphism onto $\fm$, the homogeneous maximal ideal of $R$.
\end{example}

\begin{remark}
  One can localize at the homogeneous maximal ideal
  $\fm\coloneqq R^{\geqslant 1}$ and get an example of local
  Gorenstein domain of Krull dimension one that provides a
  counterexample to the conjecture of Huneke and Wiegand.

  A commutative algebra analogue of a conjecture of Auslander and
  Reiten~\cite{Auslander/Reiten:1975a} states that over any Gorenstein
  ring $R$, if a finitely generated $R$-module $M$ satisfies
  $\Ext_R^i(M,M)=0$ for $i\ge 1$, then $M$ is free. For the ring $R$
  and ideal $I$ in Example~\ref{ex:the-main-example} one can verify
  that
  \[
    \Ext_R^2(I,I)\ne 0\,.
  \]
  Thus $I$ is not a counterexample to the Auslander-Reiten conjecture.
\end{remark}

\subsection*{An ansatz leading to Definition~\ref{defn:W}}
Let $R$ be a ring as in Definition~\ref{defn:the-ring-R} and set
$I=(t,\alpha\, t)$, for some element $\alpha$ in $L\setminus K$. As
explained in the proof of Proposition~\ref{prop:rigid}, the ideal $I$
is rigid precisely when there is an equality
\begin{equation}
  \label{eq:rigidity}
  J\cap \alpha J = \alpha J^2\,,\qquad \text{where $J\coloneqq ((t):_R \alpha\, t)$.}    
\end{equation}
Consider the following $K$-subspaces
\[
  V \coloneqq \alpha^{-1}W\cap W\quad\text{and}\quad E \coloneqq
  \alpha^{-1}H\cap H\,.
\]
It is straightforward to verify that
\[
  J = Vt \oplus Et^2 \oplus_{i\geqslant 3} Lt^i\,,
\]
and hence condition \eqref{eq:rigidity} translates to
\begin{equation}
  \label{eq:V-constraints}
  V\cap \alpha V = (0)\,, \quad E\cap \alpha E = \alpha V^2\,, \quad\text{and}\quad L =  VE\,.    
\end{equation}
Keep in mind that for any $K$-subspace $U$ of $L$, one has
$L=\alpha U$ if and only if $L=U$. The first condition means that
$V\oplus \alpha V\subseteq W$ so a first guess is to take
\[
  W\coloneqq V\oplus \alpha V\,.
\]
Since $\rank_KW$ should be $n/2$ for Proposition~\ref{prop:gorenstein}
to apply, $\rank_KV=n/4$ must hold, which explains condition (1) in
Definition~\ref{defn:W}.

Then since $V=\alpha^{-1}W\cap W$, it follows that that the sum
\[
  V +\alpha V + \alpha^2 V
\]
is direct, which explains condition (2) in Definition~\ref{defn:W}.

Now suppose that $L=K[\alpha]$ is as in Definition~\ref{defn:W}. In
the notation introduced there, one has $H=L_{n-2}$, so
\[
  E = \alpha^{-1}H\cap H = L_{n-3}\,.
\]
Thus the two conditions on the right in \eqref{eq:V-constraints}
translate to conditions (3) and (4) in Definition~\ref{defn:W}.

\section{Verification with Macaulay2}
\label{sec:M2}

\noindent In this section we provide \emph{Macaulay2} code that we
used to verify some of the assertions about
Example~\ref{ex:the-main-example}.

\subsection*{Example~\ref{ex:the-main-example} revisited}
The \emph{Macaulay2} \cite{M2} script below verifies that the subspace
$V$ in \ref{ex:the-main-example} satisfies the conditions in
Definition \ref{defn:W}. It is available from \cite{HWm2:2026}.

{\small \begin{verbatim} -- Setting up K and L

K = QQ; L = K[a]/ideal(a^24-2);

-- Defining a function to compute the standard coordinates
-- of a polynomial in L

polyToList = p -> ( apply(24, i -> coefficient(a^i,p)) );

-- Setting up the subspace V of L

gensV = {1_L, a - a^15, a^2 - a^12 - a^16, a^4 - a^14, 
    a^6 - a^13 + a^20, a^9 + a^16 + a^19};

V = image transpose matrix apply( gensV, b -> polyToList(b) );
rank V -- rank is 6
 
-- Setting up alpha times V and alpha^2 times V
 
gensaV  = apply( gensV, b-> a*b );
aV = image transpose matrix apply( gensaV, b -> polyToList(b) );
rank aV --  output: 6

gensa2V = apply( gensV, b-> a^2*b );
a2V = image transpose matrix apply( gensa2V, b -> polyToList(b) );
rank a2V -- output: 6

-- Verifying condition 3.1(1)

rank(V+aV+a2V) == rank V + rank aV + rank a2V -- output: true

-- Verifying condition 3.1(2)

gensV2 = flatten apply( 6,i -> apply( i+1, j -> (gensV#i)*(gensV#j) ));
V2 = image transpose matrix apply( gensV2, b -> polyToList(b) );
rank V2 -- output: 21

basisL20 = apply( 21, i -> a^i );
L20 = image transpose matrix apply( basisL20, b -> polyToList(b) );
V2 == L20 -- output: true

-- Verifying condition 3.1(3)

gensVL21 = flatten apply( gensV, b -> apply( 22, i -> a^i*b ));
VL21 = image transpose matrix apply( gensVL21, b -> polyToList(b) );
rank VL21 -- output: 24, so it is all of L
\end{verbatim}}

Next we realize the ring from Example \ref{ex:the-main-example} as a
quotient of a polynomial algebra.  Set
$Q = \mathbb{Q}[x_1,\ldots,x_6,y_1,\ldots,y_6]$ and consider the
canonical ring homomorphism $\phi\colon Q \to L[t]$ that maps the
indeterminates to the basis for $Wt = (V \oplus \alpha V)t$ so that
\begin{gather*}
  x_1 \mapsto t\,, \ldots\,, x_6 \mapsto (\alpha^9 + \alpha^{16} + \alpha^{19})\\
  y_1 \mapsto \alpha t\,,\ldots\,, y_6 \mapsto \alpha(\alpha^9 +
  \alpha^{16} + \alpha^{19})t\,.
\end{gather*}
By construction, the image of $\phi$ in degree $1$ is $R_1$, and since
$R$ and $Q$ are standard graded, the kernel of $\phi$ is the defining
ideal of $R$ as a quotient of $Q$.  Recalling that the Hilbert series
of $R$ is $1 + 12t + 23t^2 + \sum_{i=3}^\infty 24t^i$ and
$\rank_K\mathrm{Sym}^2(Q_1) = \binom{13}{2} = 78$, one sees that the
kernel of $\phi$ is generated by $78-23 = 55$ quartics.

The script {\small \begin{verbatim}

-- Setting up the polynomial algebra

Q = QQ[x_1..x_6,y_1..y_6];

-- Setting up the basis for Wt = (V + aV)t 

Lt = L[t];

gensWt = apply( gensV|gensaV, b -> b*t );

-- Define the canonical map phi: Q -> L[t] that maps the 
-- 12 indeterminates to the basis for Wt

phi = map(Lt,Q,gensWt);

I = ker phi;

-- Set up the 78 - 23 = 55 relations in Q to define R

rels = { x_2^2 - y_1^2 - 2*y_1*y_4,  
    x_2*x_6 + 2*x_1^2 - 36*x_1*x_2 + 72*x_1*x_3 + 42*x_1*x_4 - 2*x_1*x_5
  + 36*x_1*y_1 - 74*x_1*y_2 + 4*x_1*y_3 + 72*x_1*y_4 + 12*x_1*y_5 +
  x_1*y_6 + 2*y_1^2 + 2*y_1*y_2 - 48*y_1*y_3 + 2*y_1*y_4 - 12*y_1*y_6
  - 2*x_2*x_3 - 6*x_2*x_5 - 36*x_2*y_5 + 18*y_2^2,  
    x_2*y_1 - x_1*y_2,   
    x_2*y_2 + 24*x_1*x_2 - 48*x_1*x_3 - 28*x_1*x_4 - 24*x_1*y_1 +
  48*x_1*y_2 - x_1*y_3 - 48*x_1*y_4 - 8*x_1*y_5 - 3*y_1*y_2 +
  32*y_1*y_3 + 8*y_1*y_6 + x_2*x_3 + 4*x_2*x_5 + 24*x_2*y_5 -
  12*y_2^2,  
    x_2*y_3 + 2*x_1*x_2 - 4*x_1*x_3 - 3*x_1*x_4 - 2*x_1*y_1 + 4*x_1*y_2
  - 4*x_1*y_4 + 2*y_1*y_3 + 2*x_2*y_5 - 2*y_2^2,  
    x_2*y_4 + x_1*x_5 + 3*x_1*y_2 - x_1*y_3 - 3*y_1^2 + y_1*y_2 -
  4*y_1*y_4,  
    x_2*y_6 - 6*x_1*x_2 + 12*x_1*x_3 + 10*x_1*x_4 + 8*x_1*y_1 -
  12*x_1*y_2 + 12*x_1*y_4 + 2*x_1*y_5 - 10*y_1*y_3 - 3*y_1*y_6 -
  2*x_2*x_5 - 6*x_2*y_5 + 4*y_2^2,  
    x_3^2 - 2*x_1^2 + 4*x_1*x_2 - 8*x_1*x_3 - 6*x_1*x_4 - 4*x_1*y_1 +
  8*x_1*y_2 - 8*x_1*y_4 + 4*y_1*y_3 + 4*x_2*y_5 - 3*y_2^2,
    x_3*x_4 + x_1*x_5 + 2*x_1*y_2 - x_1*y_3 - 4*y_1^2 + y_1*y_2 -
  4*y_1*y_4,  
    x_3*x_5 + 2*x_1*x_2 - 4*x_1*x_3 + x_1*x_4 - 4*x_1*y_1 + 4*x_1*y_2 -
  4*x_1*y_4 - y_1*y_3 + x_2*y_5,  
    x_3*x_6 - 18*x_1*x_2 + 36*x_1*x_3 + 22*x_1*x_4 + 20*x_1*y_1 -
  36*x_1*y_2 + 36*x_1*y_4 + 5*x_1*y_5 - 24*y_1*y_3 - 5*y_1*y_6 -
  3*x_2*x_5 - 18*x_2*y_5 + 10*y_2^2,  
    x_3*y_1 - x_1*y_3,  
    x_3*y_2 + 2*x_1*x_2 - 4*x_1*x_3 - 3*x_1*x_4 - 2*x_1*y_1 + 4*x_1*y_2
  - 4*x_1*y_4 + 2*y_1*y_3 + 2*x_2*y_5 - 2*y_2^2,  
    x_3*y_3 - 5*x_1*x_2 - 2*x_1*x_6 + 3*x_1*y_1 - 2*x_1*y_2 + 7*x_1*y_4
  + 2*y_1^2 - 4*x_2*x_4,  
    x_3*y_4 + 22*x_1*x_2 - 44*x_1*x_3 - 27*x_1*x_4 - 22*x_1*y_1 +
  44*x_1*y_2 - 44*x_1*y_4 - 7*x_1*y_5 - 2*y_1*y_2 + 30*y_1*y_3 +
  8*y_1*y_6 + 4*x_2*x_5 + 22*x_2*y_5 - 11*y_2^2,  
    x_3*y_5 + x_1*x_6 + 2*x_1*y_2 - 2*x_1*y_3 - 4*y_1^2 + 2*y_1*y_2 -
  2*y_1*y_4,  
    x_3*y_6 + 6*x_1*x_2 - 5*x_1*x_3 + 2*x_1*x_6 - 6*x_1*y_1 + 7*x_1*y_2
  - 11*x_1*y_4 + 3*x_2*x_4 + 2*x_2*y_5,  
    x_4^2 + 3*x_1*x_2 - 6*x_1*x_3 - 2*x_1*x_4 - 3*x_1*y_1 + 6*x_1*y_2 -
  6*x_1*y_4 + 2*y_1*y_3 + 3*x_2*y_5 - 2*y_2^2,  
    x_4*x_5 - 2*x_1^2 - 24*x_1*x_2 + 48*x_1*x_3 + 28*x_1*x_4 +
  24*x_1*y_1 - 48*x_1*y_2 + 2*x_1*y_3 + 48*x_1*y_4 + 8*x_1*y_5 +
  x_1*y_6 + 2*y_1*y_2 - 32*y_1*y_3 - 8*y_1*y_6 - 2*x_2*x_3 - 4*x_2*x_5
  - 24*x_2*y_5 + 12*y_2^2,  
    x_4*x_6 + 4*x_1*x_2 - x_1*x_5 + 2*x_1*x_6 - 4*x_1*y_1 - x_1*y_2 +
  2*x_1*y_3 - 6*x_1*y_4 + y_1^2 - 2*y_1*y_2 + 3*y_1*y_4 + 2*x_2*x_4,  
    x_4*y_1 - x_1*y_4,  
    x_4*y_2 + x_1*x_5 + 3*x_1*y_2 - x_1*y_3 - 3*y_1^2 + y_1*y_2 -
  4*y_1*y_4,  
    x_4*y_3 + 22*x_1*x_2 - 44*x_1*x_3 - 27*x_1*x_4 - 22*x_1*y_1 +
  44*x_1*y_2 - 44*x_1*y_4 - 7*x_1*y_5 - 2*y_1*y_2 + 30*y_1*y_3 +
  8*y_1*y_6 + 4*x_2*x_5 + 22*x_2*y_5 - 11*y_2^2,  
    x_4*y_4 - 4*x_1*x_2 - x_1*x_6 + 4*x_1*y_1 - x_1*y_2 + 7*x_1*y_4 +
  y_1^2 - 3*x_2*x_4,  
    x_4*y_5 + 4*x_1*x_2 - 8*x_1*x_3 - 4*x_1*x_4 - 6*x_1*y_1 + 8*x_1*y_2
  - 8*x_1*y_4 + 4*y_1*y_3 + y_1*y_6 + 4*x_2*y_5 - 2*y_2^2,  
    x_4*y_6 - 38*x_1*x_2 + 76*x_1*x_3 + 47*x_1*x_4 - 2*x_1*x_5 +
  38*x_1*y_1 - 78*x_1*y_2 + 4*x_1*y_3 + 76*x_1*y_4 + 13*x_1*y_5 +
  2*x_1*y_6 + 2*y_1^2 + 2*y_1*y_2 - 52*y_1*y_3 + 2*y_1*y_4 -
  14*y_1*y_6 - 2*x_2*x_3 - 7*x_2*x_5 - 38*x_2*y_5 + 19*y_2^2,  
    x_5^2 + 4*x_1*x_2 + x_1*x_3 + 4*x_1*x_6 - 4*x_1*y_1 + 5*x_1*y_2 -
  6*x_1*y_4 - 12*y_1^2 + 2*x_2*x_4,  
    x_5*x_6 + 5*x_1*x_2 - 2*x_1*x_3 - 7*x_1*y_1 + 2*x_1*y_2 - 4*x_1*y_4
  + 2*x_2*y_5,  
    x_5*y_1 - x_1*y_5,  
    x_5*y_2 - x_2*y_5,  
    x_5*y_3 + x_1*x_6 + 2*x_1*y_2 - 2*x_1*y_3 - 4*y_1^2 + 2*y_1*y_2 -
  2*y_1*y_4,
    x_5*y_4 + 4*x_1*x_2 - 8*x_1*x_3 - 4*x_1*x_4 - 6*x_1*y_1 + 8*x_1*y_2
  - 8*x_1*y_4 + 4*y_1*y_3 + y_1*y_6 + 4*x_2*y_5 - 2*y_2^2,  
    x_5*y_5 - 144*x_1*x_2 + 288*x_1*x_3 + 168*x_1*x_4 - 2*x_1*x_5 +
  144*x_1*y_1 - 290*x_1*y_2 + 15*x_1*y_3 + 288*x_1*y_4 + 48*x_1*y_5 +
  4*x_1*y_6 + 2*y_1^2 + 15*y_1*y_2 - 192*y_1*y_3 + 2*y_1*y_4 -
  48*y_1*y_6 - 12*x_2*x_3 - 24*x_2*x_5 - 144*x_2*y_5 + 72*y_2^2,  
    x_5*y_6 + 4*x_1*x_2 + 2*x_1*x_6 - 4*x_1*y_1 + 5*x_1*y_2 + 2*x_1*y_3
  - 6*x_1*y_4 - 7*y_1^2 - 2*y_1*y_2 + 2*x_2*x_4,  
    x_6^2 + 26*x_1*x_2 - 52*x_1*x_3 - 31*x_1*x_4 - 30*x_1*y_1 +
  52*x_1*y_2 - 52*x_1*y_4 - 4*x_1*y_5 + 33*y_1*y_3 + 4*y_1*y_6 +
  4*x_2*x_5 + 26*x_2*y_5 - 14*y_2^2,  
    x_6*y_1 - x_1*y_6,   
    x_6*y_2 - 6*x_1*x_2 + 12*x_1*x_3 + 10*x_1*x_4 + 8*x_1*y_1 -
  12*x_1*y_2 + 12*x_1*y_4 + 2*x_1*y_5 - 10*y_1*y_3 - 3*y_1*y_6 -
  2*x_2*x_5 - 6*x_2*y_5 + 4*y_2^2,  
    x_6*y_3 + 6*x_1*x_2 - 5*x_1*x_3 + 2*x_1*x_6 - 6*x_1*y_1 + 7*x_1*y_2
  - 11*x_1*y_4 + 3*x_2*x_4 + 2*x_2*y_5,  
    x_6*y_4 - 38*x_1*x_2 + 76*x_1*x_3 + 47*x_1*x_4 - 2*x_1*x_5 +
  38*x_1*y_1 - 78*x_1*y_2 + 4*x_1*y_3 + 76*x_1*y_4 + 13*x_1*y_5 +
  2*x_1*y_6 + 2*y_1^2 + 2*y_1*y_2 - 52*y_1*y_3 + 2*y_1*y_4 -
  14*y_1*y_6 - 2*x_2*x_3 - 7*x_2*x_5 - 38*x_2*y_5 + 19*y_2^2,  
    x_6*y_5 + 4*x_1*x_2 + 2*x_1*x_6 - 4*x_1*y_1 + 5*x_1*y_2 + 2*x_1*y_3
  - 6*x_1*y_4 - 7*y_1^2 - 2*y_1*y_2 + 2*x_2*x_4,  
    x_6*y_6 + 4*x_1*x_2 + 4*x_1*x_3 - 2*x_1*x_6 - 4*x_1*y_1 - 6*x_1*y_2
  - x_1*y_4 - 2*y_1^2 - x_2*x_4,  
    y_2*y_3 - 6*x_1*x_2 - 2*x_1*x_6 + 6*x_1*y_1 - 2*x_1*y_2 + 9*x_1*y_4
  + 2*y_1^2 - 4*x_2*x_4,
    y_2*y_4 + 34*x_1*x_2 - 68*x_1*x_3 - 41*x_1*x_4 - 34*x_1*y_1 +
  68*x_1*y_2 - x_1*y_3 - 68*x_1*y_4 - 11*x_1*y_5 - 2*y_1*y_2 +
  46*y_1*y_3 + 12*y_1*y_6 + x_2*x_3 + 6*x_2*x_5 + 34*x_2*y_5 -
  17*y_2^2,  
    y_2*y_5 - 2*x_1*x_2 - x_1*x_6 + 2*x_1*y_1 - 2*x_1*y_3 + 3*x_1*y_4 +
  2*y_1*y_2 - 2*y_1*y_4 - x_2*x_4,  
    y_2*y_6 + 8*x_1*x_2 - 3*x_1*x_3 + 2*x_1*x_6 - 8*x_1*y_1 + 5*x_1*y_2
  - 11*x_1*y_4 - y_1*y_5 + 3*x_2*x_4 + x_2*y_5,  
    y_3^2 + 24*x_1*x_2 - 48*x_1*x_3 - 28*x_1*x_4 + 4*x_1*x_5 -
  24*x_1*y_1 + 55*x_1*y_2 - 6*x_1*y_3 - 48*x_1*y_4 - 8*x_1*y_5 -
  2*x_1*y_6 - 9*y_1^2 + 32*y_1*y_3 - 9*y_1*y_4 + 8*y_1*y_6 + 2*x_2*x_3
  + 4*x_2*x_5 + 24*x_2*y_5 - 12*y_2^2,  
    y_3*y_4 - x_1*x_2 - 2*x_1*x_4 + x_1*y_1 + 2*y_1*y_3 + y_1*y_5 -
  2*y_2^2,  
    y_3*y_5 - 16*x_1*x_2 + 32*x_1*x_3 + 16*x_1*x_4 + 16*x_1*y_1 -
  32*x_1*y_2 + 2*x_1*y_3 + 32*x_1*y_4 + 4*x_1*y_5 + x_1*y_6 +
  2*y_1*y_2 - 20*y_1*y_3 - 4*y_1*y_6 - 2*x_2*x_3 - 2*x_2*x_5 -
  16*x_2*y_5 + 8*y_2^2,  
    y_3*y_6 + 4*x_1*x_2 - 3*x_1*x_5 + 2*x_1*x_6 - 4*x_1*y_1 - 3*x_1*y_2
  + 2*x_1*y_3 - 6*x_1*y_4 + 2*x_1*y_6 + 3*y_1^2 + 5*y_1*y_4 +
  2*x_2*x_4,  
    y_4^2 + 12*x_1*x_2 - 24*x_1*x_3 - 14*x_1*x_4 + 3*x_1*x_5 -
  12*x_1*y_1 + 29*x_1*y_2 - 4*x_1*y_3 - 24*x_1*y_4 - 4*x_1*y_5 -
  x_1*y_6 - 5*y_1^2 + y_1*y_2 + 16*y_1*y_3 - 5*y_1*y_4 + 4*y_1*y_6 +
  x_2*x_3 + 2*x_2*x_5 + 12*x_2*y_5 - 6*y_2^2,  
    y_4*y_5 - 2*x_1*x_2 + x_1*x_3 + 2*x_1*y_1 - x_1*y_2 + 3*x_1*y_4 -
  2*y_1^2 + y_1*y_5 - x_2*x_4 - x_2*y_5,  
    y_4*y_6 - 3*x_1*x_2 + 10*x_1*x_3 + 10*x_1*x_4 + 3*x_1*y_1 -
  10*x_1*y_2 + 10*x_1*y_4 + 2*x_1*y_5 - 10*y_1*y_3 - y_1*y_5 -
  2*y_1*y_6 - 2*x_2*x_5 - 5*x_2*y_5 + 4*y_2^2,
    y_5^2 + 6*x_1*x_2 - 12*x_1*x_3 - 7*x_1*x_4 - 6*x_1*y_1 + 12*x_1*y_2
  - 12*x_1*y_4 + 2*x_1*y_5 + 8*y_1*y_3 - 2*x_2*x_5 + 6*x_2*y_5 -
  3*y_2^2,  
    y_5*y_6 - 80*x_1*x_2 + 160*x_1*x_3 + 96*x_1*x_4 - 2*x_1*x_5 +
  80*x_1*y_1 - 162*x_1*y_2 + 9*x_1*y_3 + 160*x_1*y_4 + 28*x_1*y_5 +
  2*x_1*y_6 + 2*y_1^2 + 10*y_1*y_2 - 108*y_1*y_3 + 2*y_1*y_4 -
  28*y_1*y_6 - 7*x_2*x_3 - 14*x_2*x_5 - 80*x_2*y_5 + 40*y_2^2,  
    y_6^2 - 24*x_1*x_2 + 48*x_1*x_3 + 28*x_1*x_4 + x_1*x_5 + 24*x_1*y_1
  - 41*x_1*y_2 + 5*x_1*y_3 + 48*x_1*y_4 + 8*x_1*y_5 - 2*x_1*y_6 -
  7*y_1^2 - 3*y_1*y_2 - 32*y_1*y_3 - 5*y_1*y_4 - 8*y_1*y_6 - 2*x_2*x_3
  - 4*x_2*x_5 - 24*x_2*y_5 + 12*y_2^2 };

-- Verify the number of relations

#rels -- output: 55

-- Verify that the relations generate the kernel of phi

ideal rels == ker phi -- output: true

R = Q/(ideal rels)

-- Verify that the Hilbert function of R is as expected

apply(6,d -> hilbertFunction(d,R)) -- output: {1, 23, 23, 24, 24, 24}

-- Set up the ideals in R corresponding to t and at
 
X = ideal(x_1); 
Y = ideal(y_1);

-- Verify the condition on these ideals from Huneke-Iyengar-Wiegand

(X:Y)*(Y:X)  == intersect(X:Y,Y:X) -- output: true

-- Notice the relation between the colon ideals observed in proof of 3.2

mingens (X:Y) -- output: | x_6 x_5 x_4 x_3 x_2 x_1 |
mingens (Y:X) -- output: | y_6 y_5 y_4 y_3 y_2 y_1 |
\end{verbatim}}

\section{Codimension two}
\label{sec:codim2}

\noindent
In this section we verify that conjecture of Huneke and Wiegand holds
when, in addition, the Gorenstein ring $R$ is of codimension at most
2; see Proposition~\ref{prop:codim2} below. A key input in the
argument is a result of Neubauer and
Saltman~\cite{Neubauer/Saltman:1994}, and we begin by recasting it
form that is convenient for our application.

\begin{proposition}
  \label{prop:BS}
  Let $k$ be a field and $A$ a finitely dimensional commutative
  $k$-algebra. Assume the $k$-algebra $A$ is two-generated. For any
  finitely generated faithful $A$-module $M$, there are inequalities
  \[
    \rank_k A\le \rank_k M \le \rank_k \End_A(M)\,.
  \]
  Moreover, if either inequality is an equality, then both become
  equalities.
\end{proposition}

There is an analogue of this result, due to Schur and Jacobson, for
arbitrary $A$ but that involves a term that is quadratic in
$\rank_kM$; see \cite{Cowsik:1993} for an elegant proof of this
result. However, this does not seem to help for applications to the
conjecture of Huneke and Wiegand.

\begin{proof}
  The hypothesis that $A$ is two-generated and $M$ is a faithful
  $A$-module is tantamount to saying that there are commuting
  $k$-linear operators $x,y\colon M\to M$ such that $A\cong k[x,y]$,
  the $k$-subalgebra of $\End_k(M)$ generated by $x,y$. Since
  $\End_A(M)$ is the $k$-subalgebra of $\End_k(M)$ consisting of
  operators that commute with both $x$ and $y$, it remains to apply
  \cite[Theorem~1.1, Lemmas~1.2 \& 1.3]{Neubauer/Saltman:1994}.
\end{proof}

Here is the avowed result about rigid ideals. We write $\edim A$ to
denote the embedding dimension of a local ring $A$.

\begin{proposition}
  \label{prop:codim2}
  Let $(R,\fm,k)$ be a Gorenstein local domain of Krull dimension
  one. Assume $\edim R\le 3$ and that there exists an element $u$ in
  $\fm\setminus \fm^2$ such that the artinian ring $R/Ru$ contains
  $k$.
    
  Every rigid ideal in $R$ is principal; equivalently, free.
\end{proposition}

\begin{proof}
  The result is well-known when $\edim R\le 2$; see \cite[Theorem
  3.1]{Huneke/Wiegand:1994}. In the rest of the proof we assume
  $\edim R=3$.

  Let $I$ be a rigid ideal in $R$. Set $S\coloneqq R/Ru$ and
  $M\coloneqq I/Iu$. Our hypothesis on means that $S$ is a local
  $k$-algebra, of embedding dimension two.

  To verify $I$ is free, it suffices to verify that the $S$-module $M$
  is free; see, for instance, \cite[??]{Bruns/Herzog:1998a}. To that
  end, since the ring $S$ is Gorenstein, it suffices to verify that
  $M$ is a faithful $S$-module, with $\rank_kS=\rank_SM$.

  Applying $\Hom_R(I,-)$ to the exact sequence of $R$-modules
  \[
    0\longrightarrow I\xrightarrow{\ u\ } I\longrightarrow
    M\longrightarrow 0
  \]
  yields the exact sequence of $R$-modules
  \[
    0\longrightarrow \End_R(I)\xrightarrow{\ u\ }\End_R(I)
    \longrightarrow \Hom_R(I,M)\longrightarrow \Ext^1_R(I,I)=0
  \]
  The equality holds because $I$ is rigid. Observing that
  $\Hom_R(I,M)\cong \End_S(M)$, one deduces from the sequence above
  that
  \[
    \End_S(M)\cong \frac{\End_R(I)}{\End_R(I)u}
  \]
  as $S$-modules. This justifies the first equality below
  \begin{equation}
    \label{eq:rankM}
    \rank_k \End_S(M) = \rank_k \left(\frac{\End_R(I)}{\End_R(I)u} \right) = \rank_kS = \rank_kM\,.
  \end{equation}
  The second and third equalities holds because $\End_R(I)$ and $I$
  are both $R$-modules of rank one, and $R$ is a domain of Krull
  dimension one; see, for instance, \cite{Bruns/Herzog:1998a}.

  Let $\overline{S}$ be the image of $S$ in $\End_S(M)$ under the
  natural map $S\to \End_S(M)$. Then $\overline{S}$ is also a
  commutative local $k$-algebra, generated by two elements. Thus
  Proposition~\ref{prop:BS} yields inequalities
  \[
    \rank_k \overline{S}\le \rank_k M \le \rank_k\End_S(M)\,,
  \]
  and also that if either inequality is an equality, then both
  are. Given \eqref{eq:rankM} it follows that there are equalities
  \[
    \rank_k S =\rank_k \overline{S} = \rank_k \End_S(M)\,.
  \]
  We deduce that the natural map is an isomorphism:
  $S\xrightarrow{\cong} \End_S(M)$, in particular, that $M$ is a
  faithful $S$-module.
\end{proof}

\begin{remark}
  A Gorenstein ring $R$ satisfying $\edim R-\dim R\le 2$ is complete
  intersection. Thus the ring $R$ in Proposition~\ref{prop:codim2} is
  a complete intersection. It remains an open question whether the
  conjecture of Huneke and Wiegand holds when $R$ is, in addition, a
  complete intersection.
\end{remark}

\bibliographystyle{amsplain} \newcommand{\noopsort}[1]{}
\providecommand{\bysame}{\leavevmode\hbox to3em{\hrulefill}\thinspace}
\providecommand{\MR}{\relax\ifhmode\unskip\space\fi MR }
\providecommand{\MRhref}[2]{%
  \href{http://www.ams.org/mathscinet-getitem?mr=#1}{#2} }
\providecommand{\href}[2]{#2}

\end{document}